\documentclass[a4paper,twoside,11pt]{article}

\usepackage{a4wide, fancyhdr, amsmath, amssymb, mathtools, yfonts}
\usepackage{mathrsfs}
\usepackage{graphicx}
\usepackage{tikz}
\usepackage[all]{xy}
\usepackage[utf8]{inputenc}
\usepackage{amsthm}
\usepackage[english]{babel}
\usepackage{chngcntr}
\usepackage{ifthen}
\usepackage{calc}
\usepackage{hyperref}
\usepackage{authblk}
\usepackage{tikz-cd}
\numberwithin{equation}{section}

\newcommand{\Z}{\mathbb{Z}}
\newcommand{\Q}{\mathbb{Q}}

\newcommand\Gal{\mathrm{Gal}}

\newcommand\ord{\mathrm{ord}}

\newtheorem{lemma}{Lemma}[section]
\newtheorem{theorem}[lemma]{Theorem}
\newtheorem{proposition}[lemma]{Proposition}

\newtheorem{mydef}[lemma]{Definition}
\newtheorem{conjecture}[lemma]{Conjecture}
\newtheorem{remark}[lemma]{Remark}

\title{\vspace{-\baselineskip}\sffamily\bfseries Malle's conjecture for fair counting functions}
\author[1]{Peter Koymans\thanks{Institute for Theoretical Studies, ETH Zurich, 8092 Zurich, Switzerland, peter.koymans@eth-its.ethz.ch}}
\author[2]{Carlo Pagano\thanks{Department of Mathematics and Statistics, Montreal, Quebec H3G 1M8, Canada, carlein90@gmail.com}}
\affil[1]{ETH Zurich}
\affil[2]{Concordia University}
\date{\today}

\begin{document}
\maketitle

\begin{abstract}
We show that the naive adaptation of Malle's conjecture to fair counting functions is not true in general.
\end{abstract}

\section{Introduction}
\subsection{Malle's conjecture}
Number field counting has a rich history in number theory, going back to at least Gauss counting quadratic extensions of $\Q$. The leading conjecture in this field is due to Malle.

\begin{conjecture}[\cite{Malle2}]
\label{cMalleDisc}
Let $G$ be a non-trivial finite group and let $k$ be a number field. Then there exists an integer $a(G) \geq 1$, an integer $b(G, k) \geq 0$ and a real number $c(G, k) > 0$ with
\begin{align}
\label{eMalleDisc}
|\{K/k : \Gal(K/k) \cong G, N_{k/\Q}(\textup{Disc}(K/k)) \leq X\}| \sim c(G, k) X^{1/a(G)} (\log X)^{b(G, k)}.
\end{align}
If we let $\textup{Cl}(g)$ be the conjugacy class of an element $g \in G - \{\textup{id}\}$ and if we write $g \sim h$ if $\textup{Cl}(g)$ and $\textup{Cl}(h)$ are equivalent under the cyclotomic action of $k$, then we have 
\begin{align}
\label{eMalleDiscb}
b(G, k) = -1 + n, 
\end{align}
where $n$ is the number of equivalence classes of $G - \{\textup{id}\}$ under $\sim$ consisting entirely of elements with minimal order in $G$.
\end{conjecture}

Although this conjecture has been exceptionally influential, numerous issues have come to light. One problem with Malle's conjecture is that it is not correct with the first counterexample due to Kl\"uners \cite{KlunersCounter}. However, there are also other undesirable features that are inherently tied with counting by discriminant that we will now discuss.

One desirable feature of a counting function is that the leading constant $c(G, k)$ is an Euler product. These type of leading constants are frequent in rational points counting, and suggest a good compatibility between global and local behavior. In the case of $S_n$, the leading constant $c(G, k)$ is conjectured to be an Euler product \cite{BhargavaMass}, in which case we say that the \emph{Malle--Bhargava} principle holds. It is however not always the case that the leading constant is an Euler product. This problem already manifests itself when dealing with quartic $D_4$-extensions, see Cohen--Diaz y Diaz--Olivier \cite{CDO} and Altug--Shankar--Varma--Wilson \cite{ASVW}. In general, it is still unclear when we expect the Malle--Bhargava principle to hold true.

Another related undesirable feature of discriminant counting is the \emph{subfield problem}. When counting by discriminant, it may happen that a positive proportion of the fields counted share a common subfield. This already happens for quartic $D_4$-extensions, and is the main underlying cause for the failure of the leading constant to be an Euler product.

\subsection{Fair counting functions}
The above reasons have led to an increasing interest in \emph{fair counting functions}, first introduced by Wood \cite{Wood}. We shall restrict ourselves to the product of ramified primes, which is the most prominent fair counting function. We consider the following naive modification of Malle's conjecture. Write $\mathfrak{f}(K/k)$ for the product of primes of $k$ that ramify in $K$.

\begin{conjecture}[Folklore adaptation of Malle's conjecture]
\label{cMalle}
Let $G$ be a non-trivial finite group and let $k$ be a number field. Then there exists an integer $b(G, k) \geq 0$ and a real number $c(G, k) > 0$ such that
\begin{align}
\label{eMalle}
|\{K/k : \Gal(K/k) \cong G, N_{k/\Q}(\mathfrak{f}(K/k)) \leq X\}| \sim c(G, k) X (\log X)^{b(G, k)}.
\end{align}
If we let $\textup{Cl}(g)$ be the conjugacy class of an element $g \in G - \{\textup{id}\}$ and if we write $g \sim h$ if $\textup{Cl}(g)$ and $\textup{Cl}(h)$ are equivalent under the cyclotomic action of $k$, then we have 
\begin{align}
\label{eMalleb}
b(G, k) = -1 + n, 
\end{align}
where $n$ is the number of equivalence classes of $G - \{\textup{id}\}$ under $\sim$.
\end{conjecture}

Maki \cite{Maki} proved this conjecture for abelian extensions, and Wood \cite{Wood} proved the same for any fair counting function and arbitrary finite sets of local conditions. It is important to emphasize that the leading constant $c(G, k)$ is an Euler product and that the subfield problem also disappears in all known cases. An additional benefit is that this counting function is much more natural from a geometric perspective occurring prominently in function field counting, which is rife with geometric techniques. 

Number field counting is intimately related to finding the distribution of class groups. Counting by discriminant also leads to problems in this setting, as first uncovered in Bartel--Lenstra \cite{BL}. Bartel--Lenstra give a counterexample to the Cohen--Martinet heuristics with the root cause being precisely the subfield problem. These reasons strongly suggests that it may be preferable to count by fair counting functions instead of the discriminant. 

\subsection{Results and conjectures}
Our main result shows that Conjecture \ref{cMalle} is not correct.

\begin{theorem}
\label{tMain}
There exists an infinite family of nilpotent groups $G$ of nilpotency class $2$ such that Conjecture \ref{cMalle} fails for the pair $(G, \Q)$. More precisely, the constant $b(G, \Q)$ in equation (\ref{eMalleb}) is too small.
\end{theorem}

Although we have restricted to $\Q$ for simplicity, it is not difficult to adapt our arguments to find many more counterexamples of a similar flavor showing that the above phenomenon persists in a wide number of settings.

Our counterexample is of a genuinely different nature than Kl\"uners' counterexample. Firstly, Kl\"uners takes $G = C_3 \wr C_2$, which is solvable but not nilpotent. Secondly, Kl\"uners' counterexample is no longer a counterexample when counting by product of ramified primes, which was historically another important motivation to count by ramified primes. In fact, the authors \cite[Section 3.2]{KPMalle1} have previously given strong heuristic evidence that Malle's original conjecture, see Conjecture \ref{cMalleDisc}, is correct for nilpotent extensions. Thirdly and perhaps most importantly, although both counterexamples proceed by fixing a cyclotomic subextension, Kl\"uners' counterexample relies on the shrinking of the cyclotomic action, while we fundamentally rely on a certain \emph{entanglement of Frobenius} that shrinks the conjugation action. 

A novel feature of our work is the first family of counterexamples when counting $G$-extensions in their regular representation. This phenomenon has never been observed for discriminant counting. All known modifications of Malle's conjecture (such as \cite{Turkelli}) predict no counterexamples when counting $G$-extensions by discriminant in their regular representation. 

We still expect the veracity of the asymptotic in equation (\ref{eMalle}), but of course not with the naive choice of $b(G, \Q)$ from equation (\ref{eMalleb}). It is at present unclear what the right choice of $b(G, \Q)$ is in general. We will however make several predictions. We have opted to restrict ourselves to nilpotent extensions as even merely making predictions for number field counting has proven to be a deceptively difficult task.

To this end, let $\phi: G \twoheadrightarrow H$ be a surjective homomorphism, let $\chi(\text{cyc}): G_\Q \rightarrow (\Z/|G| \Z)^\ast$ be the cyclotomic character and let $r: (\Z/|G| \Z)^\ast \twoheadrightarrow H$ be a surjection. We define
$$
\text{Epi}_{(H, \phi)}(G_{\Q}, G)
$$
be the set of continuous surjective homomorphisms $\psi: G_{\Q} \twoheadrightarrow G$ satisfying the equations
$$
\phi \circ \psi = r \circ \chi(\text{cyc}), \quad \Q(\psi) \cap \Q(\zeta_{|G|}) = \Q(\phi \circ \psi).
$$
We write $G \times_H (\Z/|G| \Z)^\ast \subseteq G \times (\Z/|G| \Z)^\ast$ for the subgroup consisting of pairs $(g, \alpha)$ such that $\phi(g) = r(\alpha)$. We let $G \times (\Z/|G| \Z)^\ast$ act on $\text{ker}(\phi) - \{\text{id}\}$ by sending $n$ to $gn^\alpha g^{-1}$. We restrict this action to the subgroup $G \times_H (\Z/|G| \Z)^\ast$, and we denote by
$$
b_{(H, \phi)}(G) := |(\text{ker}(\phi)-\{\text{id}\})/(G \times_{H} (\Z/|G| \Z)^\ast)|
$$
the resulting number of equivalence classes.

\begin{conjecture}[Conjecture \ref{conj1}]
\label{conjIntro}
Let $G$ be a nilpotent group. For each $H, \phi$ as above, there exists a positive constant $c_{(H, \phi)}(G)$ such that
$$
|\{\psi \in \textup{Epi}_{(H, \phi)}(G_{\Q}, G): \mathfrak{f}(\psi) \leq X\}| \sim c_{(H, \phi)}(G) \cdot X \cdot \log(X)^{b_{(H, \phi)}(G) - 1}.
$$
\end{conjecture}

\begin{conjecture}[Conjecture \ref{conj2}]
\label{conjIntro2}
Let $G$ be a nilpotent group with $|G|$ odd. Then the leading constant $c_{(H, \phi)}(G)$ satisfies the Malle--Bhargava principle.
\end{conjecture}

For the precise computation of the leading constant provided by the Malle--Bhargava principle, we refer to Conjecture \ref{conj2}. It is tempting to speculate that both conjectures are true in a much wider generality, which we shall leave as an open question. However, Conjecture \ref{conj2} does not hold when $|G|$ is even, in which case a rational correction factor is needed.

We will now explicate Theorem \ref{tMain}, in particular the construction of the infinite family $G$. We take $A_n = \Z/3\Z \oplus (\Z/9\Z)^n$. We let $\pi_0: A_n \rightarrow \Z/3\Z$ be the projection on the first $\Z/3\Z$ and we let $\pi_i: A_n \rightarrow \Z/3\Z$ be the projection on the $i$-th $\Z/9\Z$ composed with the quotient map $\Z/9\Z \rightarrow \Z/3\Z$. We introduce the $2$-cocycles
\[
\theta_i(\sigma, \tau) := \pi_0(\sigma) \cdot \pi_i(\tau) \in H^2(A_n, \Z/3\Z)
\]
and $\theta := (\theta_1, \dots, \theta_n) \in H^2(A_n, (\Z/3\Z)^n)$. Given such a $\theta$, we have a corresponding central exact sequence given by
\[
1 \rightarrow (\Z/3\Z)^n \rightarrow G_n \rightarrow A_n \rightarrow 1.
\]
In our proofs, we will choose the family $(G_n)_{n \geq 2}$ for the infinite family of Theorem \ref{tMain}. In the process we will also prove a special case of Conjectures \ref{conjIntro} and \ref{conjIntro2}.

\begin{theorem}[Theorem \ref{th: counterexample}]
\label{tGn}
Let $G_n$ be as above and write $q_n : G_n \rightarrow A_n$ for the natural quotient map. Let $H_n := \Z/3\Z$ and let $\phi := \pi_0 \circ q_n$. Let $\rho: G_\Q \rightarrow \Z/3\Z$ be any character such that the fixed field of the kernel equals $\Q(\zeta_9 + \zeta_9^{-1})$. Then we have
\[
\sum_{\substack{\psi \in \textup{Epi}(G_\Q, G_n) \\ \pi_0 \circ q_n \circ \psi = \rho \\ \mathfrak{f}(\psi) \leq X}} 1 \sim 27^n \cdot \frac{X  \cdot (\log X)^{\alpha - 1}}{3 \cdot \Gamma(\alpha)} \cdot c_0,
\]
where $\alpha := \frac{9^n - 1}{3} + \frac{27^n - 1}{6}$ and $c_0$ equals the conditionally convergent Euler product
\[
c_0 := \prod_p \left(1 + \frac{(9^n - 1)\mathbf{1}_{p \equiv 4, 7 \bmod 9} + (27^n - 1)\mathbf{1}_{p \equiv 1 \bmod 9}}{p}\right) \left(1 - \frac{1}{p}\right)^\alpha.
\]
\end{theorem}

One pleasant feature is that $\ker(\phi) \cong (\Z/9\Z)^n \oplus (\Z/3\Z)^n$ is abelian despite the fact that $G_n$ is not. It is no coincidence that the logarithmic exponent matches exactly the logarithmic exponent when counting $(\Z/9\Z)^n \oplus (\Z/3\Z)^n$-extensions. In fact, once one fixes the subextension $\Q(\zeta_9 + \zeta_9^{-1})$, we have the remarkable property that any $G_n/[G_n, G_n]$-extension lifts to a $G_n$-extension. This is the crux of our counterexample stemming from Lemma \ref{lemma: always lift}. Once we have proven Theorem \ref{tGn} in Section \ref{sLifts}, the computation of $b(G_n, \Q)$ in Section \ref{sConjugacy} immediately gives Theorem \ref{tMain}.

In fact, it is plausible that one can use \cite[Theorem 1.1]{AD} or \cite[Theorem 1.3.3]{DY} to prove Conjectures \ref{conjIntro} and \ref{conjIntro2} as soon as $\ker(\phi)$ is abelian, with substantial work required to deduce explicitly from those works the leading constant $c_0$ and the logarithmic exponent $\alpha$. We have opted to give a different proof to keep our work completely self-contained.

We remark that if $H$ and $\phi$ are trivial, then $b_{(H, \phi)}(G)$ equals the right hand side of equation (\ref{eMalleb}). In particular, we predict that the full Malle--Bhargava principle holds when counting tame, nilpotent $G$-extensions with $|G|$ odd. We emphasize that, once one accounts for the counterexamples found in Theorem \ref{tMain}, all good properties of fair counting functions are restored again: the leading constant is an Euler product and the subfield problem disappears.

\subsection{Comparison with previous results and conjectures}
Malle's original conjecture has been proven over $\Q$ for a handful of specific groups namely
\begin{itemize}
\item Davenport--Heilbronn \cite{DH} for cubic $S_3$-extensions,
\item Bhargava \cite{Bhargava1, Bhargava2} for quartic $S_4$-extensions and quintic $S_5$-extensions,
\item Belabas--Fouvry \cite{BelabasFouvry} and Bhargava--Wood \cite{BhargavaWood} for sextic $S_3$-extensions,
\item Wright \cite{Wright} for abelian extensions,
\item Masri--Thorne--Tsai--Wang \cite{MTTW}, building on earlier work of Wang \cite{Wang}, for direct products $G \times A$ with $G \in \{S_3, S_4, S_5\}$ and $A$ abelian,
\item Cohen--Diaz y Diaz--Olivier \cite{CDO} for quartic $D_4$-extensions,
\item Fouvry--Koymans \cite{FouKoy} for nonic Heisenberg extensions,
\item Kl\"uners \cite{KlunersWreath} for certain wreath products,
\item Koymans--Pagano \cite{KPMalle1} for nilpotent groups such that all minimal order elements are central.
\end{itemize}
In many situations we have upper and lower bounds for number field counting instead of asymptotics. This comprises a vast literature as well, including work of Alberts \cite{Alberts}, Kl\"uners--Malle \cite{KM}, Ellenberg--Venkatesh \cite{EV}, Couveignes \cite{Couveignes} and Lemke Oliver--Thorne \cite{LOT}.

Of particular relevance to our work is the recent spur of other modifications of Malle's conjecture. This was initiated by T\"urkelli \cite{Turkelli} for discriminant counting, who modified Conjecture \ref{cMalleDisc} to account for Kl\"uners' counterexample. Our prediction for $b_{(H, \phi)}(G)$ is a direct parallel of T\"urkelli's modification. 

Gundlach \cite[Conjecture 1.5]{Gundlach} proposed a variant of Malle's conjecture where one counts by multiple fair invariants. He avoids our counterexamples by demanding that this invariant lies in the range $(\delta X, X)$ for some $\delta > 0$. Our conjecture however does not impose this condition, which necessitates a more thorough treatment of the logarithmic exponent $b_{(H, \phi)}(G)$. 

Darda--Yasuda \cite[Conjecture 1.3.1]{DY} propose a far reaching conjecture for counting points on stacks by a wide class of height functions. Their conjecture is similar in spirit than ours but substantially less precise: our conjecture is more precise under what circumstances the logarithmic exponent may increase and also more precise about the leading constant (with no prediction being made in Darda--Yasuda). We remark that none of the aforementioned works \cite{DY, Gundlach, Turkelli} make a prediction about the leading constant $c_{(H, \phi)}(G)$. 

\subsection{Overview of the paper}
The key result for our counterexample is Lemma \ref{lemma: always lift}. This lemma shows that the group $G_n$ has the key property that all $G_n/[G_n, G_n]$-extensions lift to a $G_n$-extension once one fixes a suitable cyclotomic subextension. We will exploit Lemma \ref{lemma: always lift} to show in Section \ref{sLifts} that the count of $\psi \in \text{Hom}(G_\Q, G_n)$, containing this suitably constructed cyclotomic subextension, equals the number of abelian $(\Z/9\Z)^n \oplus (\Z/3\Z)^n$-extensions. In Section \ref{sConjugacy} we will compute the naive Malle constant $b(G_n, \Q)$. These results immediately give Theorem \ref{tMain}. In our final Section \ref{sConjecture} we motivate Conjecture \ref{conjIntro} and \ref{conjIntro2}.

\subsection*{Acknowledgements}
The first author gratefully acknowledges the support of Dr. Max R\"ossler, the Walter Haefner Foundation and the ETH Z\"urich Foundation. The authors would like to thank Ratko Darda for fruitful discussions.

\section{The construction}
\label{sGroup}
We fix in the rest of the paper an algebraic closure $\Q^{\text{sep}}$ of $\Q$ and for each place $v$ an algebraic closure $\Q_v^{\text{sep}}$ of $\Q_v$. We furthermore fix an embedding
$$
i_v: \Q^{\text{sep}} \rightarrow \Q_v^{\text{sep}},
$$
providing us with an embedding
$$
i_v^\ast:G_{\Q_v} \rightarrow G_{\Q}.
$$
We denote by $\mathcal{G}(3)$ the pro-$3$ completion of a profinite group $\mathcal{G}$. We fix a choice of the cyclotomic character
$$
\chi_{\text{cyc}}(3):G_{\Q} \twoheadrightarrow \mathbb{Z}_3.
$$
For each prime number $p$ congruent to $1$ modulo $3$, we fix an element $\sigma_p' \in G_{\Q_p}$ such that its image in $G_{\Q_p}(3)$ is a topological generator of the inertia subgroup. We denote by $\sigma_p$ the image of $\sigma_p'$ in $G_{\Q}(3)$ by applying $i_p^\ast$ followed by the natural projection map $G_{\Q} \rightarrow G_{\Q}(3)$. We fix any element $\sigma_3$ of $G_{\Q}(3)$ with $\chi_{\text{cyc}}(3)(\sigma_3) = 1$ and coming from the inertia subgroup of $G_{\Q_3}$ in the manner just described above. 

We denote by 
$$
\mathfrak{G} := \{\sigma_p : p \equiv 1 \bmod 3\} \cup \{\sigma_3\} \subseteq G_{\Q}(3).
$$
For each $p$ congruent to $1$ modulo $3$, we put $\chi_p: G_{\Q} \rightarrow \Z/3\Z$ to be the unique character which ramifies only at $p$ and with $\chi_p(\sigma_p) = 1$. We define $\chi_3$ to be the reduction modulo $3$ of $\chi_{\text{cyc}}(3)$. Furthermore, for each $p$ congruent to $1$ modulo $9$, we put $\psi_p: G_{\Q} \rightarrow \Z/9\Z$ to be the unique character which ramifies only at $p$ and with $\psi_p(\sigma_p) = 1$. We set $\psi_3$ to be the reduction modulo $9$ of $\chi_{\text{cyc}}(3)$. Then the set 
$$
\{\chi_p : p \equiv 0, 1 \bmod 3\}
$$
is a basis for $\text{Hom}(G_{\Q}(3), \Z/3\Z)$, which is dual to $\mathfrak{G}$. It follows that $\mathfrak{G}$ is a minimal set of topological generators for $G_{\Q}(3)$. 

For a continuous homomorphism $\psi: G_{\Q} \rightarrow \mathcal{G}$, where $\mathcal{G}$ is a profinite group, we denote by 
$$
\Q(\psi) := (\Q^{\text{sep}})^{\text{ker}(\psi)}. 
$$
In case $\mathcal{G}$ is a finite group, we denote by $\mathfrak{f}(\psi)$ the product of the finite rational primes ramifying in $\Q(\psi)/\Q$. Likewise, for $\psi: G_{\Q_p} \rightarrow \mathcal{G}$, we define $\mathfrak{f}(\psi)$ to be $p$ in case $\psi$ is ramified and $1$ in case $\psi$ is unramified. 

Consider $A_n = \Z/3\Z \oplus (\Z/9\Z)^n$. Write $\pi_0: A_n \rightarrow \Z/3\Z$ for the projection on the first $\Z/3\Z$ and write $\pi_i: A_n \rightarrow \Z/3\Z$ for the projection on the $i$-th $\Z/9\Z$ followed by the quotient map $\Z/9\Z \rightarrow \Z/3\Z$. We define
\[
\theta_i(\sigma, \tau) := \pi_0(\sigma) \cdot \pi_i(\tau) \in H^2(A_n, \Z/3\Z)
\]
and $\theta := (\theta_1, \dots, \theta_n) \in H^2(A_n, (\Z/3\Z)^n)$. Given such a $\theta$, we may attach a central extension
\[
1 \rightarrow (\Z/3\Z)^n \rightarrow G_n \rightarrow A_n \rightarrow 1.
\]
We will show, for sufficiently large $n$, that $G_n$ is a counterexample.

From now on we will consider all our abelian groups as discrete $G_\Q$-modules with trivial action. Given $\phi \in \text{Hom}(G_\Q, A_n)$, we write $\theta_{i, \phi} \in H^2(G_\Q, \Z/3\Z)$ and $\theta_\phi \in H^2(G_\Q, (\Z/3\Z)^n)$ for the inflation of $\theta_i$ and $\theta$ to $G_\Q$ using $\phi$. The following lemma is the crux of the counterexample. Informally speaking, it shows that any homomorphism $\phi: G_\Q \rightarrow A_n$ satisfying $\pi_0 \circ \phi \in \{\chi_3, 2\chi_3\}$, i.e. $\Q(\pi_0 \circ \phi) = \Q(\zeta_9 + \zeta_9^{-1})$, lifts to $G_n$. This will be the source of the abundance of $G_n$-extensions. 

\begin{lemma} 
\label{lemma: always lift}
Let $\phi \in \textup{Hom}(G_\Q, A_n)$ be such that $\Q(\pi_0 \circ \phi) = \Q(\zeta_9 + \zeta_9^{-1})$. Then there exists a homomorphism $\psi: G_\Q \rightarrow G_n$ lifting $\phi$.
\end{lemma}

\begin{proof}
Note that $G_n$ is isomorphic to $(\Z/3\Z)^n \times_\theta A$, where multiplication is given by
\[
(c_1, a_1) \ast_\theta (c_2, a_2) = (c_1 + c_2 + \theta(a_1, a_2), a_1 + a_2).
\]
Any lift $\psi: G_\Q \rightarrow G_n$ of $\phi$ is of the shape $(\psi', \phi)$ for some map $\psi': G_\Q \rightarrow (\Z/3\Z)^n$. Imposing that $\psi$ is a homomorphism means precisely that $\theta_\phi$ is trivial in $H^2(G_\Q, (\Z/3\Z)^n)$, which is the case if and only if each $\theta_{i, \phi}$ is trivial in $H^2(G_\Q, \Z/3\Z)$.

Therefore it suffices to show that $\theta_{i, \phi}$ is trivial in $H^2(G_\Q, \Z/3\Z)$ for each $i$. By class field theory we have an injection
\[
H^2(G_\Q, \Z/3\Z) \rightarrow \bigoplus_v H^2(G_{\Q_v}, \Z/3\Z).
\]
Now let $v$ be a place of $\Q$. We will check that the restriction of $\theta_{i, \phi}$ is trivial at $v$. If $v$ is real, then $H^2(G_{\Q_v}, \Z/3\Z) = 0$. If $v = (3)$, then we also have $H^2(G_{\Q_v}, \Z/3\Z) = 0$ by a well-known result of Shafarevich. 

We define $K_i$ to be the fixed field of $\pi_{0, i} \circ \phi$, where $\pi_{0, i}: A_n \rightarrow (\Z/3\Z)^2$ is the homomorphism given by 
\[
\pi_{0, i}(a) = (\pi_0(a), \pi_i(a)). 
\]
If $v$ is unramified in $K_i$, then the restriction of $\theta_{i, \phi}$ to $G_{\Q_v}$ factors through the maximal unramified extension of $\Q_v$, therefore giving a class in $H^2(\hat{\Z}, \Z/3\Z) = 0$.

It remains to treat finite places $v$, coprime to $3$, that are ramified in $K_i$. Recall that $\theta_{i, \phi}$ is the class of the $2$-cochain $c(\sigma, \tau)$ given by
\[
(\sigma, \tau) \mapsto \pi_0(\phi(\sigma)) \cdot \pi_i(\phi(\tau)).
\]
We claim that $\pi_0(\phi(\sigma))$ is the zero map when restricted to $G_{\Q_v}$. This implies that $c(\sigma, \tau)$ is also the zero map when restricted to $G_{\Q_v}$. In particular, $\theta_{i, \phi}$ is trivial in $H^2(G_{\Q_v}, \Z/3\Z)$, and thus the lemma is a consequence of the claim.

In order to prove the claim, observe that our assumptions imply that $v$ ramifies in $\pi_i \circ \phi$ but not in $\pi_0 \circ \phi = \rho$, which is ramified only at $3$. Since $v \neq (3)$ and since $\pi_i \circ \phi$ lifts to a $\Z/9\Z$-extension, class field theory over $\Q$ shows that $v \equiv 1 \bmod 9$. But then $v$ splits completely in $\Q(\zeta_9 + \zeta_9^{-1})$, which gives the claim.
\end{proof}

\section{Finding many lifts}
\label{sLifts}
Call $q_n: G_n \rightarrow A_n$ the natural quotient map.

\begin{mydef} 
\label{def:bad set}
We define $\mathcal{G}_{n, \textup{bad}}$ to be the set of tuples $(v_g)_{g \in G_n - \{\textup{id}\}}$ satisfying the following conditions
\begin{itemize}
\item $v_g$ is a positive squarefree integer for every $g \in G_n - \{\textup{id}\}$;
\item $v_g$ and $v_h$ are coprime for every $g, h \in G_n - \{\textup{id}\}$ with $g \neq h$;
\item if $p \mid v_g$ for some $g \in G_n - \{\textup{id}\}$ satisfying $\pi_0(q_n(g)) = 0$, then $p \equiv 1 \bmod \ord(g)$;
\item we have
\[
\prod_{\substack{g \in G_n - \{\textup{id}\} \\ \pi_0(q_n(g)) \neq 0}} v_g = 3.
\]
\end{itemize}
\end{mydef}

\begin{lemma}
\label{lOrder}
Let $g \in G_n$ and suppose that $q_n(g) \neq 0$. Then $\ord(g) = \ord(q_n(g))$.
\end{lemma}

\begin{proof}
Write $G_n = (\Z/3\Z)^n \times_\theta A_n$, write $g = (c, q_n(g))$ and write $m = \ord(q_n(g))$. Then a calculation using the group law on $G_n$ given by $\theta$ shows that
\[
g^m = \left(\sum_{j = 1}^{m - 1} \theta(q_n(g), q_n(g)^j), 0\right).
\]
Now observe that
\[
\sum_{j = 1}^{m - 1} \theta(q_n(g), q_n(g)^j) = \left(\sum_{j = 1}^{m - 1} \pi_0(q_n(g)) \cdot \pi_1(q_n(g)^j), \dots, \sum_{j = 1}^{m - 1} \pi_0(q_n(g)) \cdot \pi_n(q_n(g)^j)\right).
\]
Since we have $\sum_{j = 1}^{m - 1} \pi_i(q_n(g)^j) = \pi_i(q_n(g)) \sum_{j = 1}^{m - 1} j = 0$, the lemma follows.
\end{proof}

Following the Koymans--Pagano parametrization technique \cite{KPMalle1} yields the following key result. For the convenience of the reader we give a direct proof of this special case from scratch. 

\begin{theorem}
\label{tParGn}
There is a bijection $\textup{Par}$ between $\mathcal{G}_{n, \textup{bad}}$ and the subset of $\psi \in \textup{Hom}(G_\Q, G_n)$ such that $\Q(\pi_0 \circ q_n \circ \psi) = \Q(\zeta_9 + \zeta_9^{-1})$. Writing $\mathfrak{f}(\psi)$ for the product of ramified primes of a homomorphism $\psi$, we have
\begin{align}
\label{eProd}
\mathfrak{f}(\textup{Par}((v_g)_{g \in G_n - \{\textup{id}\}})) = \prod_{g \in G_n - \{\textup{id}\}} v_g.
\end{align}
Furthermore, the homomorphism $\textup{Par}((v_g)_{g \in G_n - \{\textup{id}\}}): G_\Q \rightarrow G_n$ is surjective if and only if
\begin{align}
\label{eEpi}
\langle q_n(\{g \in G_n - \{\textup{id}\} : v_g \neq 1\}) \rangle = A_n.
\end{align}
\end{theorem}

\begin{proof}
Let us define $\text{Hom}_{n, \text{bad}}$ to be the set of continuous homomorphisms $\psi: G_{\Q}(3) \rightarrow G_n$ with the property that
$$
\Q(\pi_0 \circ q_n \circ \psi) = \Q(\zeta_9 + \zeta_9^{-1}).
$$
We begin by constructing a map $\text{Ev}: \text{Hom}_{n, \text{bad}} \rightarrow \mathcal{G}_{n, \text{bad}}$ in the following manner. Given $\psi \in \text{Hom}_{n, \text{bad}}$, we consider the vector
$$
(v_g)_{g \in G_n - \{\text{id}\}}
$$
consisting of positive squarefree numbers divisible only by primes congruent to $0, 1$ modulo $3$ uniquely defined through the property
$$
p \mid v_g \Longleftrightarrow \psi(\sigma_p) = g.
$$
Let us check that $\text{Ev}$ indeed maps $\text{Hom}_{n, \text{bad}}$ to $\mathcal{G}_{n,\text{bad}}$. Since $\psi$ is a function, it follows that the vector $(v_g)_{g \in G_n - \{\text{id}\}}$ consists of pairwise coprime squarefree integers by construction. In other words, the first two points of Definition \ref{def:bad set} are taken care of. Let us verify the third point and distinguish for that purpose two cases. Suppose first that $q_n(g) = 0$. Then $\text{ord}(g) = 3$, so that the condition becomes $p \equiv 1 \bmod 3$, which is automatically satisfied as primes $p \equiv 2 \bmod 3$ are unramified in $3$-extensions. Suppose now that $q_n(g) \neq 0$. Then by Lemma \ref{lOrder} we know that $\text{ord}(g) = \text{ord}(q_n(g))$. Hence we need to show that $p \equiv 1 \bmod \text{ord}(q_n(g))$. The map $q_n \circ \psi \circ i_p^\ast$ induces a continuous homomorphism
$$
G_{\Q_p} \rightarrow G_{\Q_p}^{\text{ab}} \rightarrow A_n,
$$
sending $\sigma_p'$ to $q_n(g)$. But the order of the image of $\sigma_p'$ in an abelian extension of $\Q_p$ is always a divisor of $p - 1$ so that $p \equiv 1 \bmod \text{ord}(q_n(g))$ as desired.

Let us now show that the vector $(v_g)_{g \in G_n-\{\text{id}\}}$ satisfies the fourth point of Definition \ref{def:bad set}. Observe that for each prime $p \equiv 1 \bmod 3$ we have $\pi_0 \circ q_n \circ \psi(\sigma_p) \in \{\chi_3(\sigma_p), 2 \cdot \chi_3(\sigma_p)\} = \{0\}$. It follows that the variables $v_g$, with $\pi_0(q_n(g)) \neq 0$, are all equal to $1$ or $3$. Since we have shown that they are pairwise coprime, at most one of them equals $3$, namely $\psi(\sigma_3)$. This gives the desired fourth point of Definition \ref{def:bad set}.

We will now show that equation (\ref{eProd}) holds. Indeed, $3$ is certainly both a divisor of $\mathfrak{f}(\psi)$ and $\prod_{g \in G_n-\{\text{id}\}} v_g$, since the extension ramifies at $3$ and we have just verified the fourth bullet point of Definition \ref{def:bad set}. For a prime $p \equiv 1 \bmod 3$, we have that the ramification index of $\Q^{\text{ker}(\psi)}/\Q$ at $p$ equals precisely the order of $\psi(\sigma_p)$. Primes congruent to $2$ modulo $3$ are always unramified in a $3$-extension.  

We now observe that $\text{Ev}(\psi)$ determines the value of $\psi$ on $\mathfrak{G}$, which is a set of topological generators, and therefore $\text{Ev}(\psi)$ determines $\psi$. In other words, the map $\text{Ev}$ is injective. Furthermore, the fact that $\mathfrak{G}$ is a set of topological generators gives at once that 
$$
\text{im}(\psi) := \langle \psi(\mathfrak{G}) \rangle = \langle \{g \in G_n-\{\text{id}\} : v_g \neq 1\} \rangle.
$$
In particular, $\psi$ is surjective if and only if $G_n = \langle \{g \in G_n - \{\text{id}\} : v_g \neq 1\} \rangle$. Recall that a set generates a nilpotent group if and only if it generates the abelianization. Then the elementary observation that $A_n$ is the abelianization of $G_n$ yields equation (\ref{eEpi}). 

We are only left with showing that $\text{Ev}$ is surjective. We will proceed in $3$ steps. Let $(v_g)_{g \in G_n-\{\text{id}\}}$ be in $\mathcal{G}_{n,\text{bad}}$. \\
\emph{Step 1:} Define $\psi^{\text{ab}} := (\psi_i^{\text{ab}})_{i = 0}^n: G_{\Q} \rightarrow A_n$ using the formula
$$
\psi_i^{\text{ab}} := 
\begin{cases}
\pi_0(q_n(g_0)) \cdot \chi_3 &\text{for } i = 0 \\
\sum_{g \in G_n - \{\text{id}\}} \sum_{p \mid v_g} \mu_i(g) \cdot \psi_p &\text{for } 1 \leq i \leq n,
\end{cases}
$$
where $g_0$ is the unique element of $G_n - \{\text{id}\}$ with $v_{g_0} = 3$, where $\mu_i: G_n \rightarrow A_n \rightarrow \Z/9\Z$ is the projection map and where $\cdot$ denotes the usual multiplication in $\Z/3\Z$ respectively $\Z/9\Z$. \\ 
\emph{Step 2:} We have $\pi_0 \circ \psi^{\text{ab}} = \pi_0(q_n(g_0)) \cdot \chi_3$ by construction. Therefore we have that $\psi^{\text{ab}}$ can be lifted to a homomorphism $\psi:G_{\Q} \rightarrow G_n$ thanks to Lemma \ref{lemma: always lift}. The choice of such a lift consists precisely of the choice of a vector of continuous $1$-cochains $(\phi_i)_{i = 1}^n$ with each $\phi_i: G_\Q \rightarrow \Z/3\Z$ fulfilling the property
$$
d\phi_i(\sigma, \tau) = \theta_i(\psi^{\text{ab}}(\sigma), \psi^{\text{ab}}(\tau)) = \theta_{i, \psi^{\text{ab}}}(\sigma, \tau),
$$
where $d\phi_i(\sigma, \tau) = \phi_i(\sigma \tau) - \phi_i(\sigma) - \phi_i(\tau)$. Indeed, such a vector completes the map $\psi^{\text{ab}}$ into a set-theoretic map $((\phi_i)_{i = 1}^n, \psi^{\text{ab}}): G_{\Q} \rightarrow G_n$, which is a group homomorphism precisely owing to the cocycle equation. Conversely, given a lift $\psi$, we have that $\phi_i := \rho_i \circ \psi$ provides the desired $1$-cochains, where $\rho_i: G_n \rightarrow \Z/3\Z$ is the projection on the $i$-th $\Z/3\Z$ in $(\Z/3\Z)^n \times_\theta A$ (note that $\rho_i$ is not a group homomorphism).

We next define
$$
\phi_i(\text{clean}) := \phi_i - \phi_i(\sigma_3) \cdot \chi_3 - \sum_{p \equiv 1 \bmod 3} \phi_i(\sigma_p) \cdot \chi_p,
$$
which now vanishes at all the elements of $\mathfrak{G}$ (observe that the sum is finite, since $\phi_i$ vanishes on all but finitely many $\sigma_p$ by continuity). \\
\emph{Step 3:} We now define
$$
\phi_i(\text{twist}) : =\phi_i(\text{clean}) + \sum_{g \in G_n - \{\text{id}\}} \rho_i(g) \cdot \sum_{p \mid v_g} \chi_p.
$$
This gives us a homomorphism
$$
\psi(\text{twist}) := ((\phi_i(\text{twist}))_{i = 1}^n, \psi^{\text{ab}}): G_{\Q} \rightarrow G_n
$$
satisfying by construction that $\psi(\text{twist})(\sigma_p) = g$ if and only if $p \mid v_g$. Therefore
$$
\text{Ev}(\psi(\text{twist})) = (v_g)_{g \in G_n-\{\text{id}\}}.
$$
This shows that the map $\text{Ev}$ is also surjective. Putting $\text{Par} := \text{Ev}^{-1}$ finishes the proof.
\end{proof}

\noindent The properties (\ref{eProd}) and (\ref{eEpi}) are the core features of the Koymans--Pagano parametrization method \cite{KPMalle1}. Using Granville--Koukoulopoulos \cite[Theorem 1]{GK}, we get the following result.

\begin{theorem} 
\label{th: counterexample}
We have
\[
\sum_{\substack{\psi \in \textup{Hom}(G_\Q, G_n) \\ \pi_0 \circ q_n \circ \psi \in \{\chi_3, 2\chi_3\} \\ \mathfrak{f}(\psi) \leq X}} 1 \sim 2 \cdot 27^n \cdot \frac{X  \cdot (\log X)^{\alpha - 1}}{3 \cdot \Gamma(\alpha)} \cdot c_0,
\]
where
\[
\alpha := \sum_{\substack{g \in G_n - \{\textup{id}\} \\ \pi_0(q_n(g)) = 0}} \frac{1}{\varphi(\textup{ord}(g))} = \frac{3^n - 1}{2} + 3^n \cdot \left(\frac{9^n - 3^n}{6} + \frac{3^n - 1}{2}\right) = \frac{9^n - 1}{3} + \frac{27^n - 1}{6}
\]
and $c_0$ equals the conditionally convergent Euler product
\[
c_0 := \prod_p \left(1 + \frac{(9^n - 1)\mathbf{1}_{p \equiv 4, 7 \bmod 9} + (27^n - 1)\mathbf{1}_{p \equiv 1 \bmod 9}}{p}\right) \left(1 - \frac{1}{p}\right)^\alpha.
\]
The same result is true if $\textup{Hom}(G_\Q, G_n)$ is replaced by $\textup{Epi}(G_\Q, G_n)$.
\end{theorem}

\begin{proof}
We will first prove the $\textup{Hom}(G_\Q, G_n)$ case. By Theorem \ref{tParGn} we have
\[
\sum_{\substack{\psi \in \textup{Hom}(G_\Q, G_n) \\ \pi_0 \circ q_n \circ \psi \in \{\chi_3, 2\chi_3\} \\ \mathfrak{f}(\psi) \leq X}} 1 = \sum_{\substack{(v_g)_{g \in G_n - \{\text{id}\}} \in \mathcal{G}_{n, \text{bad}} \\ \prod_{g \in G_n - \{\text{id}\}} v_g \leq X}} 1.
\]
Recall that
\[
\prod_{\substack{g \in G_n - \{\textup{id}\} \\ \pi_0(q_n(g)) \neq 0}} v_g = 3.
\]
Write $T_n$ for the subset of $g \in G_n - \{\text{id}\}$ with $\pi_0(q_n(g)) = 0$. Since there are $2 \cdot 27^n$ elements $g \in G_n - \{\textup{id}\}$ with $\pi_0(q_n(g)) \neq 0$, we obtain that
\[
\sum_{\substack{(v_g)_{g \in G_n - \{\text{id}\}} \in \mathcal{G}_{n, \text{bad}} \\ \prod_{g \in G_n - \{\text{id}\}} v_g \leq X}} 1 = 2 \cdot 27^n  \sum_{\substack{(v_g)_{g \in T_n} \\ \prod_{g \in T_n} v_g \leq \frac{X}{3} \\ p \mid v_g \Rightarrow p \equiv 1 \bmod \ord(g)}} \mu^2\left(\prod_{g \in T_n} v_g\right).
\]
Writing $m := \prod_{g \in T_n} v_g$, this transforms the sum into
\[
2 \cdot 27^n \sum_{\substack{(v_g)_{g \in T_n} \\ \prod_{g \in T_n} v_g \leq \frac{X}{3} \\ p \mid v_g \Rightarrow p \equiv 1 \bmod \ord(g)}} \mu^2\left(\prod_{g \in T_n} v_g\right) = 2 \cdot 27^n \sum_{m \leq X/3} f(m),
\]
where $f(m)$ is the multiplicative function supported on squarefree integers and given on primes by
\[
f(p) = (9^n - 1)\mathbf{1}_{p \equiv 4, 7 \bmod 9} + (27^n - 1)\mathbf{1}_{p \equiv 1 \bmod 9}
\]
thanks to Lemma \ref{lOrder}. The average of $f$ on primes is equal to $\alpha$, and the theorem now follows from \cite[Theorem 1]{GK}.

To deal with $\textup{Epi}(G_\Q, G_n)$, let $S$ be a subset of $T_n$ and consider the subsum
\begin{align}
\label{eSurj}
N_1(X, S) := \sum_{\substack{(v_g)_{g \in T_n} \\ \prod_{g \in G_n - \{\text{id}\}} v_g \leq \frac{X}{3} \\ p \mid v_g \Rightarrow p \equiv 1 \bmod \ord(g) \\ v_g = 1 \Longleftrightarrow g \in S}} \mu^2\left(\prod_{g \in T_n} v_g\right).
\end{align}
This dissects the original sum into $2^{|T_n|}$ subsums. Furthermore, $S$ determines whether the resulting map will be surjective or not thanks to Theorem \ref{tParGn}. Following the argument for the homomorphism case, one may use \cite[Theorem 1]{GK} to extract an asymptotic for sums of the shape
\[
N_2(X, S) := \sum_{\substack{(v_g)_{g \in T_n} \\ \prod_{g \in G_n - \{\text{id}\}} v_g \leq \frac{X}{3} \\ p \mid v_g \Rightarrow p \equiv 1 \bmod \ord(g) \\ g \in S \Rightarrow v_g = 1}} \mu^2\left(\prod_{g \in T_n} v_g\right)
\]
for every subset $S$ of $T_n$. By \cite[Theorem 1]{GK}, we see that this sum is the largest when $S = \varnothing$, and (asymptotically) strictly smaller otherwise. But the sums in equation (\ref{eSurj}) are linear combinations of such sums. More precisely, there holds
\[
N_1(X, S) = \sum_{S \subseteq S'} (-1)^{|S'| - |S|} N_2(X, S').
\]
This includes the term $N_2(X, \varnothing)$ if and only if $S = \varnothing$, which all correspond to epimorphisms.  This proves the theorem.
\end{proof}

\section{Counting conjugacy classes}
\label{sConjugacy}
Define for $g \in G$ and $\alpha \in (\Z/\ord(g)\Z)^\ast$ the set
\[
S_{g, \alpha} := \{h \in G : h g h^{-1} = g^\alpha\}.
\]
Informally, one may think of $S_{g, \alpha}$ as the admissible Frobenius elements given that an inertia element is sent to $g$. Given $g \in G$ and a number field $k$, we may canonically identify $\Gal(k(\zeta_{\ord(g)})/k)$ as a subgroup of $(\Z/\ord(g)\Z)^\ast$. We will write this group as $T(g, k)$. The next proposition gives an explicit formula for the naive Malle constant in terms of $S_{g, \alpha}$.

\begin{proposition}
We have
\begin{align}
\label{eMalleConstant}
b(G, k) := -1 + \sum_{g \in G - \{\textup{id}\}} \frac{1}{[k(\zeta_{\ord(g)}) : k]} \sum_{\alpha \in T(g, k)} \frac{|S_{g, \alpha}|}{|G|}.
\end{align}    
\end{proposition}

\begin{proof}
Observe that
\[
\sum_{g \in G - \{\text{id}\}} \frac{1}{[k(\zeta_{\ord(g)}) : k]} \sum_{\alpha \in T(g, k)} \frac{|S_{g, \alpha}|}{|G|} = \sum_{g \in G - \{\text{id}\}} \frac{|\{\alpha \in T(g, k) : S_{g, \alpha} \neq \varnothing\}|}{[k(\zeta_{\ord(g)}) : k] \cdot |\text{Cl}(g)|}.
\]
For elements $g, h \in G$, recall that $g \sim h$ if $\text{Cl}(g)$ is equivalent to $\text{Cl}(h)$ under the cyclotomic action and also recall that $b(G, k)$ is the number of equivalence classes of $\sim$. We claim that
\begin{align}
\label{eKeyOS}
\frac{|\{\alpha \in T(g, k) : S_{g, \alpha} \neq \varnothing\}|}{[k(\zeta_{\ord(g)}) : k] \cdot |\text{Cl}(g)|} = \frac{1}{|[g]|}.
\end{align}
But this follows from the Orbit-stabilizer theorem by letting $T(g, k)$ act on the set $X := \{\text{Cl}(g^\alpha) : \alpha \in T(g, k)\}$. Indeed, first observe that
\[
|[g]| = |X| \cdot |\text{Cl}(g)|.
\]
Since the action is transitive by construction, we get for every $x \in X$
\[
|X| = |\text{Orb}(x)| = \frac{[k(\zeta_{\ord(g)}) : k]}{|\text{Stab}(x)|}.
\]
Observing that $|\text{Stab}(x)| = |\text{Stab}(\text{Cl}(g))|$ is precisely $|\{\alpha \in T(g, k) : S_{g, \alpha} \neq \varnothing\}|$, equation (\ref{eKeyOS}) follows. Therefore the naive Malle constant is equal to
\[
b(G, k) = -1 + \sum_{g \in G - \{\text{id}\}} \frac{|\{\alpha \in T(g, k) : S_{g, \alpha} \neq \varnothing\}|}{[k(\zeta_{\ord(g)}) : k] \cdot |\text{Cl}(g)|} = -1 + |(G - \{\text{id}\})/\sim|,
\]
as desired.
\end{proof}

We now calculate equation (\ref{eMalleConstant}) in the special case of $G = G_n$ and $k = \Q$. Let $g \in G_n$. If $q_n(g) = 0$, then we have $\ord(g) = 3$ and
\begin{align}
\label{eCenterCount}
S_{g, \alpha} =
\begin{cases}
G_n & \text{if } \alpha \equiv 1 \bmod 3 \\
\varnothing & \text{otherwise.}
\end{cases}
\end{align}
Now suppose that $q_n(g) \neq 0$. If $h \in S_{g, \alpha}$, we have by definition
\[
h g h^{-1} = g^\alpha.
\]
Applying $q_n$ to the above expression and recalling Lemma \ref{lOrder}, we see that $S_{g, \alpha} = \varnothing$ unless $\alpha \equiv 1 \bmod \ord(g)$. Fixing $g$, we will now compute $S_{g, 1 \bmod \ord(g)}$, which is precisely the set of $h \in G_n$ commuting with $g$. 

Observe that the map $f_g: A_n \rightarrow (\Z/3\Z)^n$ given by lifting $a \in A$ to an element $h \in G_n$ and then computing
\[
h g h^{-1} g^{-1}
\]
is a well-defined homomorphism, i.e. does not depend on the choice of lift. 
Writing out the multiplication rule for $G_n$ given by $\theta$ explicitly, we see that this homomorphism equals
\[
a \mapsto \left(\pi_0(a) \cdot \pi_i(q_n(g)) - \pi_0(q_n(g)) \cdot \pi_i(a)\right)_{1 \leq i \leq n}.
\]
We will now distinguish two cases. If $\pi_0(q_n(g)) \neq 0$, then the homomorphism $f_g$ is surjective. If instead $\pi_0(q_n(g)) = 0$, then the image of the homomorphism $f_g$ has dimension $0$ if $\pi_i(q_n(g)) = 0$ for all $i$ and dimension $1$ otherwise.

Note that the homomorphism $f_g$ depends only on $q_n(g)$. Furthermore, we have the key identity
\begin{align}
\label{eNCentralCount}
|S_{g, 1 \bmod \ord(g)}| = 3^n \cdot |\ker(f_g)| = 
\begin{cases}
3 \cdot 9^n & \text{if } \pi_0(q_n(g)) \neq 0 \\
27^n & \text{if } \pi_0(q_n(g)) = 0 \text{ and } \exists i : \pi_i(q_n(g)) \neq 0 \\
3 \cdot 27^n & \text{if } \pi_0(q_n(g)) = 0 \text{ and } \forall i : \pi_i(q_n(g)) = 0.
\end{cases}
\end{align}
Therefore we split the sum
\[
\sum_{g \in G_n - \{\text{id}\}} \frac{1}{\varphi(\ord(g))} \sum_{\alpha \in (\Z/\ord(g)\Z)^\ast} \frac{|S_{g, \alpha}|}{|G_n|} = \sum_{g \in G_n - \{\text{id}\}} \frac{1}{\varphi(\ord(g))} \cdot \frac{|S_{g, 1 \bmod \ord(g)}|}{|G_n|}
\]
in four pieces, namely the piece where $q_n(g) = 0$, the piece where $\pi_0(q_n(g)) \neq 0$, the piece where $\pi_0(q_n(g)) = 0$ and $\pi_i(q_n(g)) \neq 0$ for some $i$, and the piece where $\pi_0(q_n(g)) = 0$, $q_n(g) \neq 0$ and $\pi_i(q_n(g)) = 0$ for all $i$. The contribution from $q_n(g) = 0$ equals
\begin{align}
\label{eCount1}
\sum_{\substack{g \in G_n - \{\text{id}\} \\ q_n(g) = 0}} \frac{1}{\varphi(\ord(g))} \cdot \frac{|S_{g, 1 \bmod \ord(g)}|}{|G_n|} = \frac{3^n - 1}{2}
\end{align}
by equation (\ref{eCenterCount}). Using Lemma \ref{lOrder}, we see that the number of elements of $G_n$ with $\pi_0(q_n(g)) \neq 0$ with order $3$ is $2 \cdot 3^n \cdot 3^n$, while the number of elements of order $9$ is $2 \cdot 3^n \cdot (9^n - 3^n)$. The contribution from $\pi_0(q_n(g)) \neq 0$ becomes
\begin{align}
\label{eCount2}
\sum_{\substack{g \in G_n - \{\text{id}\} \\ \pi_0(q_n(g)) \neq 0}} \frac{1}{\varphi(\ord(g))} \cdot \frac{|S_{g, 1 \bmod \ord(g)}|}{|G_n|} &= \left(2 \cdot 3^n \cdot \frac{9^n - 3^n}{6} + 2 \cdot 3^n \cdot \frac{3^n}{2}\right) \cdot \frac{1}{3^n} \nonumber \\
&= 2 \cdot \frac{9^n - 3^n}{6} + 3^n
\end{align}
by equation (\ref{eNCentralCount}). Finally, we treat the contribution from $\pi_0(q_n(g)) = 0$ but $q_n(g) \neq 0$. Firstly, if $\pi_i(q_n(g)) = 0$ for all $i$, then $\ord(g) = 3$. There are $3^n \cdot (3^n - 1)$ such elements, and they contribute
\begin{align}
\label{eCount3}
\sum_{\substack{g \in G_n - \{\text{id}\} \\ q_n(g) \neq 0 \\ \pi_0(q_n(g)) = 0 \\ \forall i: \pi_i(q_n(g)) = 0}} \frac{1}{\varphi(\ord(g))} \cdot \frac{|S_{g, 1 \bmod \ord(g)}|}{|G_n|} = 3^n \cdot \frac{3^n - 1}{2}
\end{align}
to the total thanks to equation (\ref{eNCentralCount}). Now suppose that $\pi_i(q_n(g)) \neq 0$ for some $i$. Such $g$ have order $9$ and there are $3^n \cdot (9^n - 3^n)$ such $g$. This yields
\begin{align}
\label{eCount4}
\sum_{\substack{g \in G_n - \{\text{id}\} \\ q_n(g) \neq 0 \\ \pi_0(q_n(g)) = 0 \\ \exists i : \pi_i(q_n(g)) \neq 0}} \frac{1}{\varphi(\ord(g))} \cdot \frac{|S_{g, 1 \bmod \ord(g)}|}{|G_n|} = 3^n \cdot \frac{9^n - 3^n}{6} \cdot \frac{1}{3}
\end{align}
once more due to equation (\ref{eNCentralCount}). Adding up the contributions from (\ref{eCount1}), (\ref{eCount2}), (\ref{eCount3}) and (\ref{eCount4}) gives the following theorem.

\begin{theorem}
For all $n \geq 1$ there holds
\begin{align}
\label{eNaiveTotal}
b(G_n, \Q) + 1 = \frac{3^n - 1}{2} + 2 \cdot \frac{9^n - 3^n}{6} + 3^n + 3^n \cdot \frac{3^n - 1}{2} + 3^n \cdot \frac{9^n - 3^n}{6} \cdot \frac{1}{3}.
\end{align}
\end{theorem}

Observe that the logarithmic exponent of Theorem \ref{th: counterexample} is strictly larger than $b(G_n, \Q)$ for all $n \geq 2$. This immediately gives Theorem \ref{tMain}.

\section{A modified Malle conjecture}
\label{sConjecture}
Let $G$ be a finite nilpotent group and suppose that we have a diagram
\[
\begin{tikzcd}[row sep = small]
G \arrow[two heads]{dr}{\phi} & \\
& H \\
(\Z/|G| \Z)^\ast \arrow[two heads]{ur}{r} &
\end{tikzcd}
\]
For the rest of this section, $\chi(\text{cyc}): G_{\Q} \rightarrow (\Z/|G| \Z)^\ast$ denotes the cyclotomic character. We now define
$$
\text{Epi}_{(H, \phi)}(G_{\Q}, G)
$$
to be the set of continuous surjective homomorphisms $\psi:G_{\Q} \twoheadrightarrow G$ satisfying the equations
$$
\phi \circ \psi = r \circ \chi(\text{cyc}), \quad \Q(\psi) \cap \Q(\zeta_{|G|}) = \Q(\phi \circ \psi),
$$
i.e. we are only considering those $\psi$ with a fixed wildly ramified cyclotomic subextension. In particular, we have the diagram
\[
\begin{tikzcd}[row sep = small]
& G \arrow[two heads]{dr}{\phi} & \\
G_\Q \arrow[two heads]{ur}{\psi} \arrow[two heads]{dr}{\chi(\text{cyc})} & & H \\
& (\Z/|G| \Z)^\ast \arrow[two heads]{ur}{r} &
\end{tikzcd}
\]
We denote by $G \times_H (\Z/|G| \Z)^\ast \subseteq G \times (\Z/|G| \Z)^\ast$ the subgroup consisting of pairs $(g, \alpha)$ satisfying
$$
\phi(g) = r(\alpha).
$$
The group $G \times (\Z/|G| \Z)^\ast$ acts on $\text{ker}(\phi) - \{\text{id}\}$ by sending $n$ to $gn^\alpha g^{-1}$. Restricting this action to $G \times_H (\Z/|G| \Z)^\ast$, we denote by
$$
b_{(H, \phi)}(G) := |(\text{ker}(\phi)-\{\text{id}\})/(G \times_{H} (\Z/|G| \Z)^\ast)|
$$
the size of the quotient space. We propose the following conjecture. 

\begin{conjecture} 
\label{conj1}
Let $G$ be a nilpotent group. For each $H, \phi$ as above, there exists a positive constant $c_{(H, \phi)}(G)$ such that
$$
|\{\psi \in \textup{Epi}_{(H, \phi)}(G_{\Q}, G): \mathfrak{f}(\psi) \leq X\}| \sim c_{(H, \phi)}(G) \cdot X \cdot \log(X)^{b_{(H, \phi)}(G) - 1}.
$$
\end{conjecture}

In the following remark we compare this conjecture with the naive adaptation of Malle's conjecture and with Theorem \ref{th: counterexample}. 

\begin{remark} 
\label{remark 1}
\begin{enumerate}
\item[(a)] We can recover Malle's original exponent as follows. If one takes $H = \{\textup{id}\}$, one trivially has that $b_{(H, \phi)}(G) = b(G)$. In particular, Conjecture \ref{conj1} predicts that one can rescue Malle's conjecture in case one considers the family of extensions that are linearly disjoint from $\Q(\zeta_{|G|})$.

\item[(b)] We have restricted ourselves to maps $r: (\Z/|G| \Z)^\ast \twoheadrightarrow H$, but in principle one could consider maps $r: (\Z/|G|^j \Z)^\ast \twoheadrightarrow H$ for every $j \geq 2$ as well. Using that powering with elements $\alpha \equiv 1 \bmod |G|$ is the identity map on $G$, one can check that this does not lead to higher logarithmic exponents than the ones in our conjecture.

\item[(c)] We can recover the exponent in Theorem \ref{th: counterexample} as follows. Let $H$ be the order $3$ quotient of $(\Z/9\Z)^\ast$. Fix an identification $i: \Z/3\Z \rightarrow H$ and let $\phi := i \circ \pi_0 \circ q_n$. It is easy to verify that $\alpha$ in Theorem \ref{th: counterexample} equals exactly $b_{(H, \phi)}(G_n)$. 
\end{enumerate}
\end{remark}

We remark that one may reinterpret the exponents $b_{(H, \phi)}(G)$ as an adaptation, to the product of ramified primes, of T\"urkelli's modification of Malle's conjecture \cite{Turkelli}. Indeed, observe that the fibered product $G \times_H (\Z/|G| \Z)^\ast$ certainly contains $\text{ker}(\phi) \times \{1\}$. Hence the $G \times_H (\Z/|G| \Z)^\ast$-equivalence relation is a further equivalence relation on $\text{ker}(\phi)$-conjugacy classes in $\text{ker}(\phi)$. This further equivalence relation is obtained by acting on a $\text{ker}(\phi)$-conjugacy class via a pair $(g, \alpha)$ with $\phi(g) = r(\alpha)$. A moment's reflection shows that this comes down to the twisted $(\Z/|G| \Z)^\ast$-action on the set of $\text{ker}(\phi)$-conjugacy classes of $\text{ker}(\phi)$ given in \cite[page 198]{Turkelli}. 

\begin{remark}
A nilpotent group $G$ is always the product of its $p$-Sylow subgroups $G_p$. So the example in Theorem \ref{th: counterexample} might give the misleading impression that $b_{(H, \phi)}(G) > b(G, \Q)$ can only occur by fixing at $G_p$ some character ramified at $p$. However, choosing the central extension
$$
0 \rightarrow (\Z/2\Z)^n \rightarrow G_n \rightarrow (\Z/2\Z)^{n + 1} \times (\Z/3\Z)^n \rightarrow 0,
$$
given by the cocycles 
$$
\theta_i(\sigma, \tau) := \pi_0(\sigma) \cdot \pi_i(\tau)
$$
for $1 \leq i \leq n$, also leads to examples. Here $\pi_i$ denotes the projection map on the $i$-th copy of $\Z/2\Z$. Indeed, one may take $H := \Z/2\Z, \phi:=\pi_0$ and $r$ such that $\Q(r \circ \chi(\textup{cyc})) = \Q(\sqrt{-3})$. Then we have
$$
b_{(H, \phi)}(G_n) = \frac{12^n}{2} + O(6^n),
$$
while 
$$
b(G_n, \Q) = \frac{12^n}{4} + O(6^n),
$$
hence giving an example for $n$ sufficiently large. We leave the details of this alternative example to the interested reader. 
\end{remark}

Finally, we adapt the so-called Malle--Bhargava heuristic principle \cite{BhargavaMass} within the family $\text{Epi}_{(H, \phi)}(G_{\Q}, G)$, in order to specify the leading constant $c_{(H, \phi)}(G)$ in case $|G|$ is \emph{odd}. To this end, for a prime number $p$ and for $G, H, \phi$ as above, we denote by
$$
\text{Hom}_{(H, \phi)}(G_{\Q_p},G)
$$
the set of homomorphisms $\psi: G_{\Q_p} \rightarrow G$ such that $\phi \circ \psi = r \circ \chi(\text{cyc}) \circ i_p^\ast$. The Malle--Bhargava principle states that if one writes the following Euler product
$$
F(s) := \prod_p \left(\frac{1}{|\text{ker}(\phi)|} \cdot \sum_{\psi \in \text{Hom}_{(H, \phi)}(G_{\Q_p}, G)} \mathfrak{f}(\psi)^{-s}\right),
$$
as a Dirichlet series
$$
F(s) := \sum_{n \geq 1} \frac{f(n)}{n^s},
$$
then one expects the asymptotic
$$ 
\sum_{n \leq X} f(n) \sim |\{\psi \in \text{Epi}_{(H, \phi)}(G_{\Q}, G): \mathfrak{f}(\psi) \leq X\}|. 
$$
With this principle in mind, let us now compute the left hand side. For $g \in G$ and $\alpha \in (\Z/|G| \Z)^\ast$, we define
$$
S_{(H, \phi)}(g, \alpha) := \{h \in G : hgh^{-1} = g^{\alpha} \text{ and } \phi(h) = r(\alpha)\}.
$$
For a profinite group $\mathcal{G}$, we denote by $\mathcal{G}(p)$ the pro-$p$ completion of $\mathcal{G}$. Likewise, for a continuous homomorphism $\varphi: \mathcal{G}_1 \rightarrow \mathcal{G}_2$ of profinite groups, we denote by $\varphi(p)$ the induced map between pro-$p$ completions. We define $\mathcal{G}(\text{non-}p)$ to be the product of the pro-$q$ completions of $\mathcal{G}$ as $q$ runs over prime divisors of $|G|$ not equal to $p$. Given a continuous homomorphism $\varphi: \mathcal{G}_1 \rightarrow \mathcal{G}_2$, there is a natural induced homomorphism $\varphi(\text{non-}p)$. Write $f: H \rightarrow H(p)$ and $g: H \rightarrow H(\text{non-}p)$ for the natural surjective maps so that the product map $(f, g)$ is an isomorphism owing to the fact that $H$ is nilpotent.

We fix for each prime number $p$ a generator $\tau_p$ of the image of the inertia subgroup of $G_{\Q_p}$ in the quotient $g \circ r \circ \chi(\text{cyc}) \circ i_p^\ast$: this is a cyclic group because tame inertia is cyclic. 

\begin{proposition}
\label{pBhargava}
Notation as immediately above, we have that
$$
\sum_{n \leq X} f(n) \sim c_{(H, \phi)}(G) \cdot X \cdot \textup{log}(X)^{b_{(H, \phi)}(G) - 1},
$$
where $c_{(H, \phi)}(G)$ is the conditionally convergent product 
$$
c_{(H, \phi)}(G) := \frac{1}{\Gamma(b_{(H, \phi)}(G))} \times \alpha_1 \times \alpha_2 \times \alpha_3,
$$
where
\begin{align*}
\alpha_1 &:= 
\prod_{p \mid \mathfrak{f}(r \circ \chi(\textup{cyc}))} \left(\frac{|\textup{ker}(\phi(p))|}{p} \right) \cdot \left(\frac{\sum_{g \in \phi(\textup{non-}p)^{-1}(\tau_p)} |S_{(H(\textup{non-}p), \phi(\textup{non-}p))}(g,p)|}{|\textup{ker}(\phi(\textup{non-}p))|}\right) \left(1-\frac{1}{p}\right)^{b_{(H, \phi)}(G)} \\
\alpha_2 &:= \hspace{-0.4cm}
\prod_{\substack{p \mid |G| \\ p \nmid \mathfrak{f}(r\circ\chi(\textup{cyc}))}} \hspace{-0.15cm} \left(1 + \frac{|\textup{ker}(\phi(p))| \left(\frac{\sum_{g \in \textup{ker}(\phi(\textup{non-}p))} |S_{(H(\textup{non-}p), \phi(\textup{non-}p))}(g, p)|}{|\textup{ker}(\phi(\textup{non-}p))|}\right)-1}{p}\right) \left(1 - \frac{1}{p}\right)^{b_{(H, \phi)}(G)} \\
\alpha_3 &:= \prod_{p \nmid |G|} \left(1 + \frac{\sum_{g \in \textup{ker}(\phi) - \{\textup{id}\}} |S_{(H, \phi)}(g, p)|}{|\textup{ker}(\phi)| \cdot p}\right)\left(1 - \frac{1}{p}\right)^{b_{(H, \phi)}(G)}. 
\end{align*}
\end{proposition}

\begin{proof}
Note that $f$ is supported on squarefree integers and multiplicative away from primes dividing $|G|$. The proposition will ultimately follow from \cite[Theorem 1]{GK}, with the exponent of $\text{log}(X)$ being the average of $f$ on primes. Let us start by showing that this equals $b_{(H, \phi)}(G)$. We compute
$$
\frac{1}{|\text{ker}(\phi)| \cdot \varphi(|G|)} \cdot \sum_{\substack{g \in \textup{ker}(\phi) - \{\text{id}\} \\ \alpha \in (\Z/|G| \Z)^\ast}} |S_{(H, \phi)}(g, \alpha)|.
$$
Observe that $|\text{ker}(\phi)| \cdot \varphi(|G|) = |G \times_H (\Z/|G| \Z)^\ast|$. We can therefore rewrite the sum as
$$
\frac{1}{|G \times_H (\Z/|G| \Z)^\ast|} \cdot \sum_{(h, \alpha) \in G \times_H (\Z/|G| \Z)^\ast} |\{g \in \ker(\phi) - \{\text{id}\} : hgh^{-1} = g^\alpha\}| = b_{(H, \phi)}(G)
$$
by Burnside's lemma. 

It remains to examine the local factors of the leading constant. In \cite{GK}, the authors do so by rewriting $F(s)$ as
$$
F(s) = \frac{F(s)}{\zeta_{\Q}(s)^{b_{(H, \phi)}(G)}} \cdot \zeta_{\Q}(s)^{b_{(H, \phi)}(G)}, 
$$
in order to obtain the leading coefficient as a conditionally convergent Euler product. This is the reason for the occurrence of the term $(1 - \frac{1}{p})^{b_{(H, \phi)}(G)}$ in our formulas. We now wish to explain the remaining contributors. 

\emph{The constant $\alpha_1$:} Suppose that $p \mid \mathfrak{f}(r \circ \chi(\text{cyc}))$. Then the local contribution is precisely
$$
\frac{1}{p \cdot |\text{ker}(\phi)|} \cdot |\text{Hom}_{(H, \phi)}(G_{\Q_p}, G)|.
$$
Recalling that these are nilpotent groups, we see that we can split the count in the numerator Sylow by Sylow. Therefore we have that 
$$
|\text{Hom}_{(H, \phi)}(G_{\Q_p}, G)| = T_p \times T_{\text{non-}p},
$$
where $T_p$ equals the number of continuous homomorphisms $\psi: G_{\Q_p} \rightarrow G(p)$ such that
$$
\phi(p) \circ \psi = f \circ r \circ \chi(\text{cyc}) \circ i_p^\ast,
$$
while $T_{\text{non-}p}$ equals the number of continuous homomorphisms $\psi: G_{\Q_p} \rightarrow G(\text{non-}p)$ such that
$$
\phi(\text{non-}p) \circ \psi = g \circ r \circ \chi(\text{cyc}) \circ i_p^\ast.
$$
Note that $\psi: G_{\Q_p} \rightarrow G(p)$ factors through $G_{\Q_p}(p)$ and recall that $G_{\Q_p}(p)$ is isomorphic to a free pro-$p$ group on $2$ generators. It follows that $T_p$ equals the number of pairs of elements in $G(p)$ having prescribed value of $\phi(p)$. Therefore
$$
T_p = |\text{ker}(\phi(p))|^2.
$$
Note that any map $\psi: G_{\Q_p} \rightarrow G(\text{non-}p)$ must factor through $G_{\Q_p}^{\text{tame}}$. Recall that 
$$
G_{\Q_p}^{\text{tame}} \simeq_{\text{top.gr.}} \left(\prod_{\ell \neq p} \Z_\ell\right) \rtimes \hat{\Z},
$$
where the topological generator $1$ of the group $\hat{\Z}$ acts by multiplication by $p$ on $\prod_{\ell \neq p} \Z_\ell$. Both groups in this direct product are pro-cyclic. Hence the cardinality
$$ 
T_{\text{non-}p}
$$
equals the number of possible choices for two fixed generators. Once we prescribe that a generator of $\prod_{\ell \neq p} \Z_{\ell}$ goes to an element $g$ of $\phi(\text{non-}p)^{-1}(\tau_p) \in G(\text{non-}p)$, we have that the generator $1$ of $\hat{\Z}$ has to be sent in $S_{(H(\text{non-}p), \phi(\text{non-}p))}(g, p) \subseteq G(\text{non-}p)$. And conversely any choice of such a pair gives rise to a valid homomorphism. This proves that
$$ 
T_{\text{non-}p} = \left(\sum_{g \in \phi(\text{non-}p)^{-1}(\tau_p)} |S_{(H(\textup{non-}p), \phi(\text{non-}p))}(g,p)|\right), 
$$
which gives us the desired conclusion on $\alpha_1$.

\emph{The constant $\alpha_2$:} Suppose that $p \mid |G|$ and $p \nmid \mathfrak{f}(r \circ \chi)$. Then splitting the local factor
$$
\frac{1}{|\text{ker}(\phi)|} \cdot \sum_{\psi \in \text{Hom}_{(H, \phi)}(G_{\Q_p}, G)} \mathfrak{f}(\psi)^{-1},
$$
into unramified and ramified $\psi$, we get precisely
$$
1 + \frac{|\{\psi \text{ ramified and } \psi \in \text{Hom}_{(H, \phi)}(G_{\Q_p}, G)\}|}{p \cdot |\ker(\phi)|}.
$$
This is because there are precisely $|\text{ker}(\phi)|$ unramified ones. Indeed, the Galois group of the maximal unramified extension of $\Q_p$ is topologically free on one generator and therefore the unramified elements of $\text{Hom}_{(H, \phi)}(G_{\Q_p}, G)$ correspond to the choices of an element of $G$ with prescribed image under $\phi$. 

Now we can use again the count of unramified classes in $\text{Hom}_{(H, \phi)}(G_{\Q_p}, G)$ to obtain that
$$
|\{\psi \text{ ramified and } \psi \in \text{Hom}_{(H, \phi)}(G_{\Q_p}, G)\}| = |\text{Hom}_{(H, \phi)}(G_{\Q_p}, G)| - |\ker(\phi)|.
$$
But we have computed in the evaluation of $\alpha_1$ that
$$
|\text{Hom}_{(H, \phi)}(G_{\Q_p}, G)| = |\textup{ker}(\phi(p))|^2 \left(\sum_{g \in \phi(\textup{non-}p)^{-1}(\tau_p)} |S_{(H(\textup{non-}p), \phi(\textup{non-}p))}(g, p)|\right).
$$
Since $p \nmid \mathfrak{f}(r \circ \chi(\text{cyc}))$, we have $r \circ \chi(\text{cyc}) \circ i_p^\ast(I_p) = \{1\}$ and thus $\phi(\textup{non-}p)^{-1}(\tau_p) = \ker(\phi(\textup{non-}p))$. This gives the desired formula. 

\emph{The constant $\alpha_3$:} Since $p$ does not divide $|G|$, we now have that all the homomorphisms will be tame. In particular, this implies that the image of a generator of tame inertia has to be in $\ker(\phi)$. The total contribution from unramified homomorphisms is again $1$, as already articulated above. The one from ramified ones comes precisely in the same way we have explained in the computation of $\alpha_1$. 
\end{proof}

With the above in mind, we are ready to make the following conjecture.

\begin{conjecture} 
\label{conj2}
Suppose $G$ is odd and nilpotent. Then Conjecture \ref{conj1} holds with 
$$
c_{(H, \phi)}(G) := \frac{1}{\Gamma(b_{(H, \phi)}(G))} \times \alpha_1 \times \alpha_2 \times \alpha_3,
$$
where $\alpha_i$ are as in Proposition \ref{pBhargava}.
\end{conjecture}

It is readily verified that Theorem \ref{th: counterexample} is a special case of Conjecture \ref{conj2}, with the choice of $G, H, \phi$ as explained in Remark \ref{remark 1}. The extra factor $2$ in Theorem \ref{th: counterexample} accounts for the fact that one may also choose another identification in Remark \ref{remark 1}. The factor $\frac{27^n}{3}$ is the local factor at $3$ and the constant $c_0$ therein is the product of the tame factors along with the factor $(\frac{2}{3})^{\alpha}$.

The conjecture can be extended also for a finite prescribed set of local conditions by modifying accordingly the local factors defining $F(s)$, namely summing $\mathfrak{f}(\psi)^{-s}$ only among the prescribed local homomorphisms $\psi$. In particular, if one runs only over tame extensions, one gets the simpler leading constant
$$
\prod_{p \mid |G|} \left(1 - \frac{1}{p}\right)^{b_{(H, \phi)}(G)} \times 
\prod_{p \nmid |G|} \left(1 + \frac{\sum_{g \in \textup{ker}(\phi) - \{\textup{id}\}} |S_{(H, \phi)}(g, p)|}{|\textup{ker}(\phi)| \cdot p}\right) \left(1 - \frac{1}{p}\right)^{b_{(H, \phi)}(G)}.
$$
We have excluded the groups of even cardinality from Conjecture \ref{conj2}, since one can prove that the same conjecture would fail already for the dihedral group on $8$ elements, and even among tamely ramified extensions. In this case the leading constant is likely to be an Euler product times a rational correction factor to account for quadratic reciprocity. If also wild extensions are considered, then even further modifications may be needed by Grunwald--Wang type of obstructions. These problems at $2$ corresponds to the unspecified constant $C(G)$ in \cite[Equation (8.6), page 17]{BhargavaMass}.

\end{document}